\documentclass[12pt]{amsart}

\usepackage[dvips]{graphics}
\usepackage{epsfig}
\usepackage{amssymb}
\usepackage{amsthm}
\usepackage{amscd}
\usepackage[mathscr]{eucal}
\usepackage{enumitem}

\usepackage{hyperref}

\newtheorem{Theorem}{Theorem}[section]
\newtheorem{theorem}{Theorem}[section]

\newtheorem{corollary}[Theorem]{Corollary}

\newtheorem{lemma}[Theorem]{Lemma}
\newtheorem{Fact}[Theorem]{Fact}
\newtheorem{fact}[Theorem]{Fact}
\newtheorem{remark/def}[Theorem]{Remark/Definition}

\newtheorem{claim}[Theorem]{Claim}

\theoremstyle{definition}

\newtheorem{remark}[Theorem]{Remark}

\newtheorem{definition}[Theorem]{Definition}

\newtheorem{notation}[Theorem]{Notation}

\newtheorem{note}[Theorem]{Note}

\newtheorem{question}[Theorem]{Question}
\newtheorem{def/rem}[Theorem]{Definition/Remark}
\newsavebox{\indbin}
\savebox{\indbin}{\begin{picture}(0,0)
\newlength{\gnu}
\settowidth{\gnu}{$\smile$} \setlength{\unitlength}{.5\gnu}
\put(-1,-.65){$\smile$} \put(-.25,.1){$|$}
\end{picture}}
\def \indo {\mathop{\smile \hskip -0.9em ^| \ }}

\newcommand{\be}{\begin{enumerate}}
\newcommand{\bd}{\begin{defn}}

\newcommand{\bt}{\begin{theorem}}
\newcommand{\bl}{\begin{lemma}}
\newcommand{\ee}{\end{enumerate}}
\newcommand{\ed}{\end{defn}}
\newcommand{\et}{\end{theorem}}
\newcommand{\el}{\end{lemma}}

\newcommand{\sm}{\setminus}

\newcommand{\CP}{{\mathcal P}}
\newcommand{\CV}{{\mathcal V}}

\newcommand{\CB}{{\mathcal B}}
\newcommand{\CC}{{\mathcal C}}

\newcommand{\CL}{{\mathcal L}}

\newcommand{\CM}{{\mathcal M}}
\newcommand{\CN}{{\mathcal N}}

\newcommand{\CT}{{\mathcal T}}

\newcommand{\CU}{{\mathcal U}}
\newcommand{\BN}{{\mathbb N}}
\newcommand{\BZ}{{\mathbb Z}}
\newcommand{\BQ}{{\mathbb Q}}
\newcommand{\BR}{{\mathbb R}}

\newcommand{\CS}{{\mathcal S}}

\newcommand{\id}{\operatorname{id}}

\newcommand{\aut}{\operatorname{Aut}}

\newcommand{\autf}{\operatorname{Autf}}

\newcommand{\gal}{\operatorname{Gal}}

\newcommand{\sdist}{\operatorname{Sd}}
\newcommand{\bsdist}{\operatorname{\widehat{S}d}}

\newcommand{\Th}{\operatorname{Th}}

\def\eq{\operatorname{eq}}

\def\cl{\operatorname{cl}}

\def\dcl{\operatorname{dcl}}

\def\bdd{\operatorname{bdd}}
\def\acl{\operatorname{acl}}

\def\ker{\operatorname{Ker}}

\def\tp{\operatorname{tp}}

\def\stp{\operatorname{stp}}
\def\ltp{\operatorname{Ltp}}

\def\x{\bar{x}}
\def\y{\bar{y}}
\def\z{\bar{z}}
\def\a{\bar{a}}
\def\b{\bar{b}}

\def\raw{\rightarrow}

\def\sm{\setminus}
\def\ob{\operatorname{Ob}}
\def\mor{\operatorname{Mor}}
\def\supp{\operatorname{supp}}
\def\Bd{\partial}

\title{Lascar groups and the first homology groups of strong types in rosy theories}

\author{Junguk \textsc{Lee}}

\address{Department of Mathematics\\ Yonsei University\\
50 Yonsei-Ro, Seodaemun-Gu\\
Seoul 120-749, Korea}

\email{ljw@yonsei.ac.kr}

\begin{document}

\begin{abstract}
For a rosy theory, we give a canonical surjective homomorphism from a Lascar group over $A=\acl^{eq}(A)$ to a first homology group of a strong type over $A$, and we describe its kernel by an invariant equivalence relation. As a consequence, we show that the first homology groups of strong types in rosy theories have the cardinalities of one or at least $2^{\aleph_0}$. We give two examples of rosy theories having non trivial first homology groups of strong types over $\acl^{eq}(\emptyset)$. In these examples, these two homology groups are exactly isomorphic to their Lascar group over $\acl^{eq}(\emptyset)$.
\end{abstract}

\maketitle

\section{Introduction}
In model theory, strong types and Lascar types are important objects to understand invariant equivalence relations in a theory. Let $T=T^{eq}$ be a theory and $\CC$ be a monster model of $T$. Let $A$ be a small set in $\CC$. We say that an equivalence relation $E$ is {\em finite} if it has finitely many $E$-classes, {\em bounded} if the number of $E$-classes are less than the cardinality of $\CC$, and {\em $A$-invariant} if for $a,b$ and $f\in\aut_A(\CC)$, $(a,b)\in E$ implies $(f(a),f(b))\in E$. We say finite tuples $a,b$ have the same {\em strong(or Shelah) type over $A$}, written $a\equiv_A^s b$ if for any $A$-definable finite equivalence relation $E$ over $\CC^{|a|}$, $(a,b)\in E$. It is well known that $a\equiv_A^s b$ if and only if $a\equiv_{\acl(A)} b$. Thus we take $\stp(a/A)$, the strong type of $a$ over $A$ as $\tp(a/\acl(A))$, and it is the orbit of $a$ under the action of $\aut_{\acl(A)}(\CC)$. So, it makes sense to consider a strong type of an infinite tuple over $A$. The Lascar types over $A$ are defined as follows : We say tuples $a,b$ possibly infinite length have the same {\em Lascar type over $A$}, written $a\equiv_A^L b$ if for any $A$-invariant bounded equivalence relation $E$ over $\CC^{|a|}$, $(a,b)\in E$. Like strong types, a well known fact is that $a\equiv_A^L b$ if and only if there is a finite sequence $(a_0,a_1,\ldots,a_n)$ in $\CC^{|a|}$ such that $a_0=a$, $a_n=b$, and $a_i\equiv_{M_i} a_{i+1}$ for a submodel $M_i$ of $\CC$ containing $A$ for each $i=0,\ldots,n-1$. So $\ltp(a/A)$, the Lascar type of $a$ over $A$ is the orbit of $a$ under the action of $\autf_{A}(\CC)$ which is the group of automorphism fixing a submodel containing $A$, and $\autf_{A}(\CC)$ is a normal subgroup in $\aut_A(\CC)$. So, it is well defined the group $\gal_L(\CC;A)=\aut_A(\CC)/\autf_A(\CC)$, called the Lascar group over $A$. One asks when $\stp=\ltp$ in $T$ and it is determined by $\gal_L(\CC;A)$ for $A=\acl(A)$.

Two notions of strong and Lascar types are also important in the context of classification theory. In stable theories, each type over a model $A$ is stationary; it has only one of unique non-forking extension over each $B$ containing $A$. And this stationarity over a model is a characterization property of stable theories. In the case when a theory has a ternary invariant independence relation satisfying symmetry, transitivity, extension, finite and local characters, and stationarity over a model, it is well known that the theory is stable.  In simple theories, the stationarity over models is substituted by type amalgamation over a model( or $3$-amalgamation); for a model $A$ and $a_1,a_2,b_1,b_2$ with $a_1\equiv_A a_2$, if $\{a_i,b_i\}$ for $i=1,2$ and $\{b_1,b_2\}$ are independent over $A$, then there is $a_3\equiv_{Ab_i} a_i$ for $i=1,2$ such that $\{a_3,b_1,b_2\}$ is independent over $A$. One generalizes $3$-amalgamation to $n$-amalgamation for $n\ge 3$, called generalized amalgamation properties. We assume $T=T^{heq}$ in the case that $T$ is simple. It is well known that each strong type over $A$ which is a type over $\acl(A)$ is stationary in a stable theory $T$, and if $T$ is simple, $a\equiv_A^L b$ if and only if $a\equiv_{\bdd(A)}b$ and each types over $\bdd(A)$ satisfies $3$-amalgamation, where $\bdd(A)$ is the set of elements in $\CC^1$ whose the cardinality of orbit under $\aut_A(\CC)$ is less than $|\CC|$. In \cite{GKK}, J. Goodrick, B. Kim, and A. Kolesnikov introduced homology groups related with generalized amalgamation property for strong types in the context of rosy theories. They computed the first homology groups for some cases in \cite{GKK}(which were all zero), and in stable theories(of course, $3$-amalgamation holds in this case), they gave an explicit description of higher homology groups in \cite{GKK1}\cite{GKK2}. It was not known much about the first homology groups in general. We supposed that the first homology group for a strong type $p$  in a rosy theory is related with a Lascar group. Indeed, in \cite{KKL}, B. Kim, S. Kim, and the author showed that for a Lascar strong type $p$, the first homology group for $p$ is always zero, so even though in a rosy theory, Lascar strong type has no $3$-amalgamation still it satisfies a kind of complicated form of amalgamation.

In this paper, instead of the original definition of the first homology group in \cite{GKK}, we work with a restricted one, called a first homology group {\em over $A$} in a strong type over $A$ for $A=\acl(A)$ and we give a canonical surjective homomorphism from the Lascar group over $A$ into these first homology groups in strong types over $A$ in rosy theories. We describe its kernel by an invariant bounded equivalence relation whose classes are described by some subgroup of the automorphism group. From this description of the kernel, we deduce that the cardinality of this first homology group for a strong type is always one or at lest $2^{\aleph_0}$. At last, we give two examples of rosy theories having a strong type over $\acl^{eq}(\emptyset)$ with a non trivial first homology group exactly isomorphic to their Lascar groups over $\acl^{eq}(\emptyset)$. From known examples in \cite{GKK}\cite{KKL} and our two examples, we conjecture that these first homology group in strong types over $A$ are isomorphic to the abelianization of Lascar group over $A=\acl(A)$ under the assumption that the algebraic closures of non-empty small sets are again models in a rosy theory $T$.\\


\medskip

We review some notions and facts from
\cite{GKK1},\cite{GKK} and \cite{KKL}. First we recall the definitions
of simplices and the corresponding homology groups introduced in
\cite{GKK1},\cite{GKK}. {\bf Throughout we work with a large saturated model $\CC(=\CC^{\eq})$ whose theory $T(=T^{eq})$ is rosy with the thorn-independence relation} $\indo$ {\bf on the small sets of $\CC$. For a small $A\subset \CC$, we denote the algebraic closure and definable closure in the home sort as $\acl_{\CC}(A)$ and $\dcl_{\CC}(A)$, and in the imaginary sort as $\acl_{\CC}^{eq}(A)$ and $\dcl_{\CC}^{eq}(A)$. If there is no risk of confusion, we shall write $\acl(A)$ and $\dcl(A)$ in both cases. We fix a small algebraically closed set $A=\acl(A)$ and $p(x)\in S(A)$ (with possibly infinite $x$)}. When we say $T$ is simple, we consider $T=T^{heq}$.

\medskip

Let $\CS_A$ denote the category, where 
\be

	\item The objects are small subsets of $\CC$ containing $A$, and 
	
	\item The morphisms are elementary maps which fix $A$ pointwise.
	
\ee
And for a finite $s \subset \omega$, the power set of $s$, $\CP(s)$ forms the category as an ordered set :
\be
	
	\item $\ob(\CP(s))=\CP(s)$, and
	
	\item For $u, v \in \CP(s)$, $\mor(u,v)=\{\iota_{u,v}\}$, where $\iota_{u,v}$ is the single inclusion map for $u\subseteq v$, or $=\emptyset$ otherwise.
	
\ee
For a functor $f:\CP(s) \to \CC_A$ and $u\subseteq v \in \CP(s)$, we
write $f^u_v:=f(\iota_{u,v})\in \mor(f(u),f(v))$ and $f^u_v(u):=f^u_v(f(u))\subseteq f(v)$.

\begin{definition}
A functor $f:\CP(s) \to \CS_A$ for some finite $s \subset \omega$ is
said to be a {\em closed independent (regular) $n$}-{\em simplex} in $p$ if 
\be
\item  $|s|=n+1$

\item $f(\emptyset) \supseteq A$; and for $i\in s$, $f(\{i\})$ is of the form $\acl(Ca)$ where $a(\models p)$ is independent with $C=f^{\emptyset}_{ \{i\} }(\emptyset)$ over $A$.

\item
For all non-empty $u\in \CP(s)$, we have
$$f(u) = \acl(A \cup \bigcup_{i\in u} f^{\{i\}}_u(\{i\}));$$
and $\{f^{\{i\}}_u(\{i\})|\ i\in u \}$ is independent over $f^{\emptyset}_u(\emptyset)$.

\ee
We say $f$ is {\em over $A$} if $f(\emptyset)=A$(so for any $u\subset s$, $f^{\emptyset}_u(\emptyset)=A$). We shall call a closed independent $n$-simplex simply by an {\em $n$-simplex}. The set $s$ is called the {\em support of $f$}, denoted by $\supp(f)$. 
\end{definition}
In this paper, we only consider simplices over $A$. {\bf We fix an enumeration of $\acl(aA)$ for each $a\in \CC^{|x|}$} such that $a,b\in \CC^{|x|}$, $a\equiv_A b$ if and only if $\acl(aA)\equiv\acl(bA)$ because in \cite{KKT}, there is a counterexample which fails generalized amalgamation properties without fixing enumeration of bounded closed set.

\begin{definition}
Let $S_{n}(p;A)$ denote the collection of all $n$-simplices over $A$ in $p$ and
$C_{n}(p;A)$ the free abelian group generated by $n$-simplices in $S_{n}(p;A)$; its elements are called $n$-{\em chains over $A$} in $p$.

A non-zero $n$-chain $c$ is uniquely written (up to permutation of $i$'s) as $c=\sum_{1\le i\le
k}\limits n_i f_i$, where $n_i$ is a non-zero integer and
$f_1,\ldots,f_k$ are distinct $n$-simplices. We call $|c|:=|n_1|+\cdots+|n_k|$ the {\em length} of the chain $c$, and
define the {\em support} of
$c$ as the union of $\supp(f_i)$'s. 
\end{definition}

\noindent We use $a,b,c,\ldots,f,g,h,\ldots,\alpha,\beta,\ldots$ to denote
simplices and chains. Now we define the boundary operators and using
the boundary operators we will define homology groups.

\begin{definition}
Let $n \geq 1$ and $0 \leq i \leq n$. The $i$-th {\em boundary
operator} $\Bd_n ^i : C_n(p;A) \raw C_{n-1}(p;A)$ is defined so
that if $f$ is an $n$-simplex with domain $\CP(s)$ with
$s = \{s_0<\cdots<s_n \}$, then
    \begin{center}
        $\Bd_n^i(f)=f\upharpoonright \CP(s\sm\{s_i\})$
    \end{center}
and extended linearly to all $n$-chains in $C_n(p)$.

The {\em boundary map} $\Bd_n
: C_n(p;A)\raw C_{n-1}(p;A)$ is defined by the rule
    \begin{center}
        $\Bd_n(c)=\sum_{0\leq i\leq n}\limits (-1)^i \Bd_n^i(c)$.
    \end{center}

We write $\Bd^i$ and $\Bd$ for $\Bd_n^i$ and $\Bd_n$,
respectively, if $n$ is clear from context.
\end{definition}

\begin{definition}
The kernel of $\Bd_n$ is denoted $Z_{n}(p;A)$, and its elements
are called ($n$-){\em cycles over $A$}. The image of $\Bd_{n+1}$ in $C_n(p;A)$
is denoted by $B_{n}(p;A)$ and its elements are called ($n$-){\em boundaries over $A$}.
\end{definition}

Since $\Bd_n\circ\Bd_{n+1}=0$, $B_n(p;A)\subseteq Z_n(p;A)$
and we can define simplicial homology groups in $p$.

\begin{definition}
The $n$-th ({\em simplicial}) {\em homology group over $A$} in $p$ is
\[H_{n}(p;A):= Z_{n}(p;A)/B_{n}(p;A).\]
\end{definition}
\noindent In \cite{GKK}, original simplicial homology groups in $p$ were defined using not only simplices over $A$ in $p$ but also other simplices in $p$. But in this paper, we consider the homology groups over $A$ in $p$.
\begin{notation}
We shall abbreviate $S_n(p;A),C_n(p;A),\ldots$ as $S_n(p),C_n(p),\ldots$ and we shall also abbreviate $H_n(p;A)$ simply as $H_n(p)$.
\end{notation}

\begin{definition}
For $n \geq 1$, an $n$-chain $c$ is called an $n$-{\em shell} if it is in the
form 
$$c= \pm\sum_{0\leq i\leq n+1}\limits (-1)^i f_i,$$
where $f_0,\cdots,f_{n+1}$ are $n$-simplices such that whenever
$0\leq i < j \leq n+1$, we have $\Bd^i f_j = \Bd^{j-1} f_i$. Specially, a $1$-shell $c$ is of the form $$c=f_{0}-f_{1}+f_{2}.$$
\end{definition}

\begin{remark} The boundary of an $2$-simplex is a $1$-shell, and the boundary of any $1$-shell is $0$.
\end{remark}

\begin{definition}
For $n\geq 0$, we say $p$ has $(n+2)$-{\em amalgamation} if any $n$-shell in $p$ is the boundary of some $(n+1)$-simplex in $p$, and $p$ has $(n+2)$-{\em complete amalgamation} (or simply $(n+2)$-CA) if $p$ has $k$-amalgamation for every $2 \leq k \leq n+2$. By extension axiom of the independence relation, whenever $f : \CP(s) \raw
\CC_A$, $g : \CP(t) \raw \CC_A \in S(p)$ and $f\upharpoonright
\CP(s\cap t) = g \upharpoonright \CP(s\cap t)$, then $f$ and $g$ can be extended to a simplex $h : \CP(s\cup t) \raw \CC_A$ in $p$. This property is called {\em strong} $2$-{\em amalgamation}.
\end{definition}

The following fact shows why the notion of shells is important.

\begin{fact}\label{funfact}\cite{GKK1},\cite{GKK}
If $p$ has $(n + 1)$-CA for some $n \geq 1$, then
    \begin{center}
    $H_n(p) = \{ [c] : c \mbox{ is an } n\mbox{-shell over }A \mbox{
      with }\supp(c)=\{0,\ldots,n+1\}\;\}$.
    \end{center}
Thus the first homology group in $p$ is generated by $1$-shells in $p$ with its support $\{0,1,2\}$.
\end{fact}

\noindent So, we have that $H_1(p)$ is trivial iff  any 1-shell of its support $\{0,1,2\}$ in
$p$ is the boundary of some 2-chain in $p$. Therefore, if $T$ is
simple, due to 3-amalgamation $H_1(p)$ is trivial. The following shows
that the same result holds in any rosy theory.
\begin{fact}\label{h1=0}\cite{KKL}
Suppose that $p$ is any Lascar strong type in a rosy theory. Then $H_1(p)=0$.
\end{fact}

We introduce a notion of type homologies in \cite{GKK}. We call types with possibly infinite sets of variables $*$-types. We fix a set $\CV$ of variables which is large enough so that all variables in $*$-types come from the set $\CV$ and $|\CC|>2^{|\CV|}$. For any $X\subset \CV$, any injective function $\sigma:\ X\rightarrow \CV$, and any $*$-type $p(\x)$ with $\x\subset X$, we let $\sigma_* p:=\{\phi(\sigma(\x)):\ \phi(\x)\in p\}$. For $A=\acl(A)$, let $\CT_A$ be the category, where
\be
	\item The objects of $\CT_A$ are all the complete $*$-types in $T$ over $A$, including a single distinguished type $p_{\emptyset}$ with no free variables;
	\item $\mor_{\CT_A}(p(\x),q(\y)$ is the set of all injective maps $\sigma:\ \x\rightarrow \y$ such that $\sigma_* p\subset q$. 
\ee
\begin{definition}
Let $A=\acl(A)$ and $p\in S(A)$. A {\em closed independent type-$n$-smplex in $p$} is a functor $f:\ \CP(s)\rightarrow \CT_A$ for $s\subset \omega$ such that

\be
	\item $|s|=n+1$.
	\item Let $w\subset s$ and $u,v\subset w$. Set $f_w^u:=f(\iota_{u,w})$. Write $\x_w$ as the variable set of $f(w)$. Then whenever $\a$ realizes the type $f(w)$ and $\a_u$, $\a_v$, and $\a_{u\cap v}$ denote subtuples corresponding to the variable sets $f_w^u(\x_u)$, $f_w^v(\x_v)$, and $f_w^{u\cap v}(\x_{u\cap v}$, then $$\a_u\indo_{A\cup \a_{u\cap v}} \a_v.$$
	\item For all non-empty $u\subset s$ and any $\a$ realizing $f(u)$, we have $\a=\acl(A\cup\bigcup_{i\in u}\a_{\{i\}}$.
	\item For $i\in s$, $f(\{i\})$ is the complete $*$-type of $\acl(AC\cup\{b\})$ over $A$, where $C$ is some realization of $f(\emptyset)$ and $b$ is some realization of a nonforking extension of $p$ to $AC$.
\ee
We say $f$ is {\em over $A$} if $f(\emptyset)=A$.
\end{definition}
Using closed independent type-functors in $p$ we define the $n$-the type homology groups over $A$ in $p$, denoted by $H_n^{t}(p;A)$. We shall write $H_n^t(p;A)$ as $H_n^t(p)$. Then for each $n$, the $n$-homology groups $H_n(p;A)$ and $H_n^t(p;A)$ are non-canonically isomorphic, which is depending on the choice of enumerations of each $*$-types of closed independent simplices in $p$.

We see a notion of chain-walk notion, which is motivated from directed walk in graph theory, in \cite{KKL}\cite{KL}. The chain-walk was used to reduce a given a $2$-chain of $1$-shell boundary to one of simple form of $2$-chain having the same $1$-shell boundary and it is useful to compute the first homology group of a strong type. Two fundamental operations were used in reducing $2$-chains to the forms of chain-walks : {\em crossing} and {\em renaming support} operations. We refer the reader to \cite{KKL} for the definitions of crossing and renaming support operations and to \cite{KKL}\cite{KL} for the detail of classification of $2$-chains. In \cite{KL}, one defined the chain-walk using notion of direct walk in graph theory, here we give the definition of chain-walk in terms of simplces.
\begin{definition}\label{chainwalk}
Let $\alpha$ be a 2-chain having the boundary $f_{12} -f_{02} +
f_{01}$.
A subchain $\beta=\sum_{i=0}^m\limits \epsilon_i b_i$ of $\alpha$ (where $\epsilon_i = \pm 1$ and   $b_i$ is a $2$-simplex, for each $i$) is called
a {\em chain-walk in $\alpha$ from $f_{01}$ to $-f_{02}$} if
        \be\item there are non-zero numbers $k_0,\ldots,k_{m+1}$ (not necessarily distinct) such that   $k_0=1$, $k_{m+1}=2$, and for $ i\leq m$, $\supp(b_i)=\{k_i,k_{i+1},0\}$;
        \item  $(\Bd \epsilon_0 b_0)^{0,1} = f_{01}$, $(\Bd \epsilon_m b_m)^{0,2}=- f_{02}$; and
        \item for $0 \le i < m$,
        $$(\Bd \epsilon_i b_i)^{0, k_{i+1}}+ (\Bd \epsilon_{i+1} b_{i+1})^{0,k_{i+1}}=0.$$
        \ee
\end{definition}
\noindent Any $2$-chain having a $1$-shell boundary is reduced to a chain-walk $2$-chain having the same boundary of the support $\{0,1,2\}$.
\begin{fact}\label{supp=3}{\cite{KKL}\cite{KL}} Applying crossing and renaming support operation to a $2$-chain $\alpha$ with the $1$-shell boundary $f_{12} -f_{02}
+f_{01}$, it is reduced to a $2$-chain $\alpha'=\sum_{i=0}^{2n}\limits (-1)^i a_i$ with $|\alpha'|\le |\alpha|$, which itself is a chain-walk from $f_{01}$ to $-f_{02}$, and $\supp(\alpha')=\{0,1,2\}$.
\end{fact}
\section{Lascar groups and the first homology groups}
In this section we show that there is a canonical epimorphism from the Lascar group $\gal_L(\CC;A)$ over $A$ into the first homology group $H_1(p)$ in $p$.

Let $f\colon \CP(s) \raw \CC_A$ be a $n$-simplex in $p$. For $u\subset s$ with $u = \{ i_0 <\ldots <i_k \}$, we shall write $f(u)=[a_0 \ldots a_k]_u$, where $a_j\models p$,
$f(u)=\acl(A,a_0\ldots a_k)$, and $\acl(a_j A)=f^{ \{ i_j \} } _u (\{i_j\})$, or we write $f(u)\equiv [a_0 \ldots a_k]_u$ by changing '$=$' to '$\equiv$'. In both cases, $\{a_0,\ldots,a_k\}$ is independent over $A$.
\subsection{Representations of 1-shells}
Given two $1$-shells $s_k=f^k_{01}+f^k_{12}-f^k_{02}$ with $k=0,1$, if $f^0_{ij}(\{i,j\})\equiv_A f^1_{ij}(\{i,j\})$ for $0\le i<j\le 2$, then the homology classes of $s_0$ and $s_1$ are same in $H_1(p)$ since $H_1^t(p)\cong H_1(p)$. From this, we introduce a notion of a representation of a $1$-shell and we describe the first homology group in $p$ using this notion.

\begin{def/rem}
Let $s=f_{01}+f_{12}-f_{02}$ be a $1$-shell such that $\supp(f_{ij})=\{i,j\}$ for $0\le i<j\le 2$. There is a quadruple $(a_0,a_1,a_2,a_3)$ in $p(\CC)^4$ such that $f_{01}(\{0,1\})\equiv[a_0 a_1]_{\{0,1\}}$, $f_{12}(\{1,2\})\equiv[a_1 a_2]_{\{1,2\}}$, and $f_{02}(\{0,2\})\equiv[a_3 a_2]_{\{0,2\}}$. We call this quadruple {\em a representation of $s$}.

Note that a representation for a $1$-shell need not be unique and it is possible that the same quadruple represents different $1$-shells even though they have the same support because given $a,b\models p$, the enumeration of $\acl(aA)$ is fixed but the enumeration of $\acl(abA)$ is not fixed.
\end{def/rem}

\begin{definition}
Let $s$ be a $1$-shell and $(a,b,c,a')$ be a representation of $s$. We call $a$ {\em an initial point}, $a'$ {\em a terminal point}, $(a,a')$ {\em an endpoint pair} of this representation.
\end{definition}

In the next theorem, we'll see that the endpoint pairs of representations determine the classes of $1$-shells in $H_1(p)$, and the group structure of $H_1(p)$ can be described by endpoint pairs.
\begin{theorem}\label{endpt_class}
Let $s_0$ and $s_1$ be $1$-shells of the support $\{0,1,2\}$. Suppose they have some representations of a same endpoint pair. Then $s_0-s_1$ is a boundary of a $2$-chain, that is, they are in the same homology class in $H_1(p)$
\end{theorem}
\begin{proof}
We assume that $A=\emptyset$. Consdider two $1$-shelss $s_k=f_{12}^k-f_{02}^k+f_{01}^k$ for $k=0,1$. Suppose $s_0$ and $s_1$ have representations $(a,b_0,c_0,a')$ and $(a,b_1,c_1,a')$ respectively. Take two independent elements $b,c\models p$ such that $bc\indo ab_0b_1c_0c_1a'$ and consider a $1$-shell $s$ of its support $\{0,3,4\}$ which is represented by $(a,b,c,a')$. Then there is a $2$-chain $\alpha=(a^0_{01}+a^0_{12}-b^0-a^0_{02})-(a^1_{01}+a^1_{12}-b^1-a^1_{02})$ where for each $k=0,1$ and $0\le i<j\le 2$, $a^k_{ij}$ and $b^k$ are $2$-simplices satisfying the followings : 

\be
	\item If $j-i=1$, $\supp(a^k_{ij})=\{i,j,3\}$, otherwise, $\supp(a^k_{02})=\{0,2,4\}$. And $\supp(b^k)=\{2,3,4\}$;
	\item $a^k_{ij}\upharpoonright \CP(\{i,j\})=f^k_{ij}$; and
	\item $a^k_{01}(\{0,1,3\})=[a b_k b]_{\{0,1,3\}}$, $a^k_{12}(\{1,2,3\})=[b_k c_k b]_{\{1,2,3\}}$, $a^k_{02}(\{0,2,4\})=[a'c_k c]_{\{0,2,4\}}$; and $b^k(\{2,3,4\})=[c_k bc]_{\{2,3,4\}}$.
\ee

\noindent and $\Bd(a^k_{01}+a^k_{12}-b^k-a^k_{02})=s_k-s$. So $\Bd\alpha=(s_0-s)-(s_1-s)=s_0-s_1$, and $s_0$ and $s_1$ are in the same homology class. 
\end{proof}

\begin{theorem}\label{endpt_gpstr}
Let $s_0$ and $s_1$ be $1$-shells of a support $\{0,1,2\}$, and let $a,a',a''\models p$ be such that $(a,a')$ and $(a',a'')$ are endpoint pairs of representations of $s_0$ and $s_1$ respectively. Then there is a $1$-shell $s$ of a support $\{0,1,2\}$ having an endpoint of representation $(a,a'')$ and $[s]=[s_0]+[s_1]$ in $H_1(p)$.
\end{theorem}
\begin{proof}
Assume $A=\emptyset$. Consider two $1$-shells of a support $\{0,1,2\}$, $s_0=f^0_{01}+f^0_{12}-f^0_{02}, s_1=f^1_{01}+f^1_{12}-f^1_{02}$. Suppose there are representations of $s_0,s_1$ such that the terminal point of one of $s_0$ and the initial point of one of $s_1$ are same. Let $a'\models p$ be the common element and $a,a''\models p$ be elements so that $(a,a')$ and $(a',a'')$ are endpoint pairs of $s_0$ and $s_1$ respectively. Let $b_0, b_1, c_0, c_1\models p$ be elements such that two quadruples $(a,b_0,c_0,a')$ and $(a',b_1,c_1,a'')$ are representations of $s_0$ and $s_1$ respectively. Consider two independent elements $d,e\models p$ with $de\indo aa'a''b_0 b_1 c_0 c_1$. Then there is a $2$-chain $\alpha=(a^0_{01}+a^0_{12}-a^0_{02})-b+(a^1_{01}+a^1_{12}-a^1_{02})$, where for $k=0,1$ and $0\le i<j\le 2$, $a^k_{ij}$ and $b$ are $2$-simplices satisfying the followings :
\be
	\item $\supp(a^k_{ij})=\{i,j,3+k\}$ and $\supp(b)=\{0,3,4\}$;
	\item $a^k_{ij}\upharpoonright \CP(\{i,j\})=f^k_{ij}$;
	\item $a^0_{01}(\{0,1,3\})=[a,b_0,d]_{\{0,1,3\}}$, $a^0_{12}(\{1,2,3\})=[b_0,c_0,d]_{\{0,2,3\}}$, $a^0_{02}(\{0,2,3\})=[a',c_0,d]_{\{0,2,3\}}$, $a^1_{01}(\{0,1,4\})=[a',b_1,e]_{\{0,1,4\}}$, $a^1_{12}(\{1,2,4\})=[b_1,c_1,e]_{\{1,2,4\}}$, $a^1_{02}(\{0,2,4\})=[a'',c_1,e]_{\{0,2,4\}}$, and $b(\{0,3,4\})=[a',d,e]_{\{0,3,4\}}$.
\ee
\noindent and $\Bd(a)=s_0+s_1-s'$, where $s'=-a^0_{01}\upharpoonright \CP(\{0,3\})+b\upharpoonright \CP(\{3,4\})+a^1_{12}\upharpoonright \CP(\{0,4\})$ is a $1$-shell of a support $\{0,3,4\}$. Using the proof of Theorem \ref{endpt_class}, we get a $1$-shell $s$ of a support $\{0,1,2\}$ having an endpoint $(a,a'')$ and $[s]=[s']$ in $H_1(p)$. Thus, there is a 2-chain $\alpha'$ having a $1$-chain $s_0+s_1-s$ as its boundary and so $[s]=[s_0]+[s_1]$.
\end{proof}
Next we consider an action of $\aut_A(\CC)$ on each $C_n(p)$ and this action induces an action of $\aut_A (\CC)$ on $H_n(p)$. From the theorem \ref{endpt_class}, this action becomes trivial on $H_n(p)$. But this triviality is very crucial in finding a connection between the Lascar group over $A$ and the first homology group in $p$.
\begin{def/rem}
We define an action of $\aut_A(\CC)$ on each $C_n(p)$. Let $\sigma\in \aut(\CC)$. For a $n$-chain $c=\sum_{i=0}^k\limits n_i f_i$, we define $$\sigma(c):=\sum_{i=0}^k\limits n_i \sigma(f_i),$$
where for a $n$-simplex $f:\CP(s)\rightarrow C_A$ with $s=\{s_0<s_1<\cdots<s_n\}$, a $n$-simplex $\sigma(f)$ is defined as follows :
\be

	\item $\sigma(f)(u):=\sigma(f(u))$ for each $u\subset s$; and
	\item $\sigma(f)(\iota_{u,v}):=\sigma\circ f(\iota_{u,v})\circ \sigma^{-1}$ for each inclusion map $\iota_{u,v}$.

\ee
Furthermore, this action commutes with $\Bd$, i.e.,
$$\Bd(\sigma(c))=\sigma(\Bd(c)).$$
So this action induces an action of $\aut_A(\CC)$ on $H_1(p)$ as follows : for each $[s]\in H_1(p)$, $\sigma([s]):=[\sigma(s)]$.
\end{def/rem}
\begin{note}
Let $s$ be a $1$-shell in $p$ and let $(a,b)$ be an endpoint pair of $s$. For each $\sigma\in\aut_A(\CC)$, $(\sigma(a),\sigma(b))$ is an endpoint pair of $\sigma(s)$.
\end{note}
\noindent Since the $n$-th type-homology group and the $n$-th homology group in $p$ are isomorphic, the action of $\aut_A(\CC)$ on $H_1(p)$ is trivial.
\begin{corollary}\label{auto_action_triviality}
Let $s$ be a $1$-shell and let $\sigma \in \aut_A(\CC)$. Then there is a $2$-chain $\alpha$ having the boundary of $s-\sigma(s)$.
\end{corollary}
\noindent We denote the ordered bracket $[a,b]$ for the class of $1$-shell $s$ in $H_1(p)$ which has an endpoint pair $(a,b)$ for $a,b\models p$. By Theorem \ref{endpt_class}, this bracket notion is well-defined. We can summarize Theorems \ref{endpt_class}, \ref{endpt_gpstr}, and Corollary \ref{auto_action_triviality} as follows : For $a,b,c\in p(\CC)$ and $\sigma\in\aut_A(\CC)$, in $H_1(p)$,
\be

	\item $[a,b]+[b,c]=[a,c]$;
	\item $[a,a]$ is the identity element;
	\item $-[a,b]=[b,a]$; and
	\item $\sigma([a,b])=[\sigma(a),\sigma(b)]=[a,b]$.

\ee
\subsection{Lascar group and the first homology groups}
Here, using the ordered bracket notion of endpoint pairs, we define a map $\psi_a$ from the automorphism group over $A$ into the first homology group in $p$ for each $a\models p$. This map is proven to be a surjective homomorphism(or epimorphism) and this map does not depending on the choice of $a\models p$. Thus we get a canonical epimorphism from $\aut_A(\CC)$ into $H_1(p)$ and we study about its kernel.

For each $a\models p$, we define a map $\psi_a$ from $\aut_A(\CC)$ to $H_1(p)$ by sending $\sigma$ into $[a,\sigma(a)]$.
\begin{theorem}\label{canonical_epi}
\be
	\item Each $\psi_a$ is a epimorphism;
	\item For $a,b\models p$, $\psi_a=\psi_b$. So we get a canonical map $\psi$ from $\aut_A(\CC)$ into $H_1(p)$.
\ee
\end{theorem}
\begin{proof}
(1) Fix $a\models p$. At first, surjectivity of $\psi_a$ comes from the fact that for $b\models p$, there is a $\sigma\in \aut_A(\CC)$ such that $\sigma(a)=b$. It is enough to show that $\psi_a$ is a homomorphism. For $\sigma,\tau\in \aut_A(\CC)$,
$$\begin{array}{c c l}
\psi(\sigma\tau)&=&[a,\sigma\tau(a)]\\
&=&[a,\sigma(a)]+[\sigma(a),\sigma\tau(a)]\\
&=&[a,\sigma(a)]+\sigma[a,\tau(a)]\\
&=&[a,\sigma(a)]+[a,\tau(a)]\\
&=&\psi_a(\sigma)+\psi_a(\tau).
\end{array}$$
So $\psi_a$ is a homomorphism.\\

\smallskip

(2) Choose $a,b\models p$. Then there is $\tau\in \aut_A(\CC)$ such that $b=\tau(a)$. For $\sigma\in\aut(\CC)$,
$$\begin{array}{c c l}
\psi_b(\sigma)&=&\psi_a(\tau^{-1}\sigma\tau)\\
&=&\psi_a(\tau^{-1})+\psi_a(\sigma)+\psi_a(\tau)\\
&=&\psi_a(\sigma).
\end{array}$$
Thus $\psi_a=\psi_b$ and we get a canonical epimorphism $\psi(=\psi_a):\ \aut_A(\CC)\rightarrow H_1(p)$ for some $a\models p$.
\end{proof}
\noindent So $H_1(p)$ is isomorphic to $\aut_A(\CC)/\ker(\psi)$ and it is need to understand the kernel of $\psi$. In \cite{KKL}, it was shown that if $p$ is a Lascar strong type, then the first homology group is zero. This fact can be restated using endpoint notion as follows :
\begin{fact}\label{Lascar_zero}{\cite{KKL}}
Let $a,b\models p$ be such that $a\equiv_A^L b$. Then any $1$-shell having endpoint pair $(a,b)$ is a boundary of a $2$-chain, i.e., $[a,b]=0$ in $H_1(p)$.
\end{fact}
\begin{definition}
\be
	\item Let $\aut_B(\CC)$ be the set of elements $\sigma\in \aut(\CC)$ fixing $B$ pointwise.
	
	\item For a group $G$, the commutator of $G$ is the subgroup of $G$ generated by $\{ghg^{-1}h^{-1}|\ g,h\in G\}$, denoted by $[G,G]$ and this is the smallest normal subgroup between normal subgroups $N$ of $G$ making $G/N$ abelian.
\ee
\end{definition}
\begin{theorem}\label{kernel_canonical_epi}
Let $N$ be the normal subgroup of $\aut_A(\CC)$ generated by automorphisms in $\autf_A(\CC)$, or in $\aut_{\acl(Aa)}$ for some $a\models p$, let $G=\aut_A(\CC)/N$, and consider the canonical quotient map $\Psi':\aut_A(\CC)\rightarrow G$. Let $\psi:\aut_A(\CC)\rightarrow H_1(p)$ be the canonical map. Then the kernel of $\psi$ contains the followings :
\be
	\item $\aut_{\acl(Aa)}(\CC)$ for each $a\models p$;
	\item $\autf_A(\CC)$; and
	\item $(\Psi')^{-1}([G,G])$.	
\ee
Specially, from the second one, $\psi$ induces a canonical epimorphism $\Psi$ from $\gal_L(\CC;A)$ into $H_1(p)$. 
\end{theorem}
\begin{proof}
(1) For any $a\models p$, $[a,a]$ is the identity element in $H_1(p)$, and since we fix an enumeration of $\acl(aA)$, $\aut_{\acl(Aa)}(\CC)$ is contained the kernel of $\psi$.\\

\smallskip

(2) It comes from Fact \ref{Lascar_zero}\\

\smallskip

(3) It comes from the fact that $H_1(p)$ is always abelian.
\end{proof}
We define an $A$-invariant equivalence relation on $p(\CC)^2$, and using this equivalence relation we describe the kernel of $\psi$. By Fact \ref{supp=3}, we can describe $1$-shells which are boundary of $2$-chains as follows :
\begin{theorem}\label{chainwalk_representation}
A $1$-shell $s$ is a boundary of $2$-chain if and only if there is a representation $(a,b,c,a')$ such that for some $n$ there is a finite sequence $(d_i)_{0\le i\le 2n+1}$ of elements in $p(\CC)$ satisfying the following conditions :
\be
	\item $d_0=a,d_{2n+1}=c$ and $d_{2i_0}=a'$ for some $0<i_0 \le n$;
	\item $\{d_{2i},d_{2i+1},b\}$ is independent for each $0\le i\le n$; and
	\item There is a bijection $m$ from $\{0,2,\cdots,2n\}\setminus \{2i_0\}$ to $\{1,3,\cdots,2n-1\}$ such that $d_{2i} d_{2i+1}\equiv d_{m(2i)+1}d_{m(2i)}$ for $0\le i\neq i_0 \le n$, and $d_{2i_0}d_{2i_0 +1}\equiv a'c$.
\ee
\end{theorem}
\begin{proof}
Let $s=f_{01}+f_{12}-f_{02}$ be a $1$-shell and let $\alpha=\sum_{i=0}^{2n}\limits (-1)^i a_i$ be a $2$-chain with $\supp(\alpha')=\{0,1,2\}$ and $\Bd(\alpha)=s$, which is a chain-walk from $f_{01}$ to $-f_{02}$. Then there are a representation $(a,b,c,a')$ and a finite sequence $(d_i)_{0\le i\le 2n+1}$ of elements in $p(\CC)$ satisfying the following conditions :
\be
	\item $d_0=a,d_{2n+1}=c$ and $d_{2i_0}=a'$ for some $0<i_0 \le n$;
	\item $\{d_{2i},d_{2i+1},b\}$ is independent for each $0\le i\le n$; and
	\item There is a bijection $m$ from $\{0,2,\cdots,2n\}\setminus \{2i_0\}$ to $\{1,3,\cdots,2n-1\}$ such that $d_{2i} d_{2i+1}\equiv d_{m(2i)+1}d_{m(2i)}$ for $0\le i\neq i_0 \le n$, and $d_{2i_0}d_{2i_0 +1}\equiv a'c$.
\ee
\end{proof}
\noindent Now we define an $A$-invariant equivalence relation on $p(\CC)$ representing the kernel of $\psi$. For $m\ge 1$, define a partial type $p^{\odot m}(x_1,\ldots,x_m)$ over $A$ as
$$\bigwedge_{i\le m} p(x_i)\wedge \bigwedge_{j<m} x_i\indo x_{i+1}.$$

\noindent Next define a relation $\sim$ on $p^{\odot 4}(\CC)$ as follows : for $(a_i,b_i,c_i,a'_i)\in p^{\odot 4}(\CC)$ and $i=0,1$, $(a_0,b_0,c_0,a'_0)\sim (a_1,b_1,c_1,a'_1)$ if $a_0b_0\equiv_A a_1b_1$, $b_0c_0\equiv_A b_1c_1$, and $a'_0c_0\equiv_A a'_1c_1$. Note that for $(a_0,b_0,c_0,a'_0)$ and $(a_1,b_1,c_1,a'_1)$ in $p^{\odot 4}(\CC)$, $(a_0,b_0,c_0,a'_0)\sim (a_1,b_1,c_1,a'_1)$ if and only if both quadruples represent a same $1$-shell. This relation is an $A$-type-defianble equivalence relation. And for each $n\ge 0$, the $(1)$, $(2)$ and $(3)$ conditions of $(a,b,c,a')$ and $(d_0,d_1,\ldots,d_{2n+1})$ in Theorem \ref{chainwalk_representation} is $A$-type-definable. We define a partial type $F_n(x,y,z,w;v_0,v_1,\ldots,v_{2n+1})$ as
$$\begin{array}{c l}
&\bigwedge_{0\le i\le 2n+1}\limits p(v_i) \wedge v_0=x \wedge v_{2n+1}=w\\
\wedge & \bigwedge_{0\le j\le n}\limits \{v_{2j},v_{2j+1},z\}\mbox{ is independent over }A\\
\wedge & \bigvee_{0\le i_0\le n}\limits [ d_{2i_0}=y\\
\wedge & \bigvee_m\limits\ (\bigwedge_{0\le i\neq i_0\le n}\limits v_{2i}v_{2i+1}\equiv_A v_{m(2i)+1}v_{m(2i)})\\
&\mbox{for each bijetion } m : \{0,2,\cdots,2n\}\setminus \{2i_0\}\rightarrow \{1,3,\cdots,2n-1\}\\
\wedge & v_{2i_0}v_{2i_0+1}\equiv_A yw].

\end{array}
$$
At last, for each $n\ge 0$, we define a partial type $E'_n(x,y)$ over $A$ as
$$\begin{array}{c l}
&p(x)\wedge p(y)\\
\wedge &\exists zwx'y'z'w'\ [( p^{\odot 4}(x,z,w,y)\wedge p^{\odot 4}(x',z',w',y')\wedge (x,z,w,y)\sim (x',z',w',y'))\\
\wedge & \exists v_0 v_1\ldots v_{2n+1}\ F_n(x',y',z',w';v_0,v_1,\ldots,v_{2n+1})]. 
\end{array}$$
The relation $E'_n(x,y)$ says that $(x,y)$ is an endpoint pair of a $1$-shell which is a boundary of $2$-chain which is a chain-walk of length $2n+1$. Take $E_n(x,y)\equiv E'_n(x,y)\wedge E'_n(y,x)$. So, for each $n\ge 0$, $E_n$ is $A$-type-definable symmetric relation. At last, define the binary relation $E(x,y)$ as $$x=y\vee \bigvee_{n\ge 0}E_n(x,y).$$
This relation is $A$-invariant, reflexive, and symmetric. By Theorem \ref{endpt_gpstr}, it is transitive and by Theorem \ref{chainwalk_representation} $E(a,b)$ if and only if $[a,b]=0$ in $H_1(p)$ for $a,b\models p$. So this relation $E$ is a desired $A$-invariant equivalence relation.\\

\smallskip

Next, we define a distance-like notion on $p(\CC)$ as follows : For $a,b\models p$,
$$
d_E(a,b):=
\begin{cases}
\min\{n|E_n(a,b)\} & \mbox{if } E(a,b)\\
\infty & \mbox{otherwise.}
\end{cases}
$$
This distance-like notion is not necessary to satisfy triangle inequality, i.e., for $a,b,c\models p$, $d_E(a,b)\le d_E(a,c)+d_E(c,b)$. But it does hold that for $a,b,c\models p$, $d_E(a,b)\le d_E(a,c)+d_E(c,b)+8$ since in the proof of Theorem \ref{endpt_gpstr} for two $1$-shells $s_0$ and $s_1$, there is a $1$-shell $s$ such that $s_0+s_1-s$ is a boundary of $2$-chain of length of $15(=2\times 8-1)$. So we can apply the results in \cite{N} and we know that if $E$ is not type-definable, then the cardinality of $H_1(p)$ is at least $2^{\aleph_0}$. And, in the Appendix A, we prove that for any bounded (type-)definable equivalence relation on a strong type, the possible cardinality of the set of equivalence classes in the strong type is one or at least $2^{\aleph_0}$. In \cite{KrT}, an invariant equivalence relation $E$ on a type-definable set $X$ is called {\em orbital equivalence relation} if there is a subgroup $\Gamma$ of $\aut(\CC)$ such that $\Gamma$ preserves classes of $E$ setwise and acts transitively on each class. By Theorem \ref{canonical_epi}, $E(a,b)$ if and only if there is $\sigma\in \ker(\psi)$ such that $\sigma(a)=b$ for $a,b\models p$. So, our equivalence relation $E$ is a orbital equivalence relation.
\begin{theorem}
\be 
	\item $E$ is an orbital equivalence relation.
	\item The cardinality of  $H_1(p)$ is zero or $\ge 2^{\aleph_0}$.
\ee
\end{theorem}
\noindent Next section, we give two examples and in two examples, we compute two first homology groups, which are non trivial and their cardinalities are exactly $2^{\aleph_0}$.

\section{Examples}
In simple theories including stable theories, the first homology group of a strong type is always zero by $3$-amalgamation. In \cite{GKK}, the first homology groups of strong types were computed for some cases and they were all zero, and they showed that in o-minimal theories, the first homology group of a strong 1-type is always trivial. Here, we give two examples of rosy theories having a non trivial first homology group in a strong type. They are the first cases to give a non trivial first homology group in a strong type.  In \cite{KKL}, B. Kim, S. Kim, and the author considered the structures in \cite{CLPZ}, $\CM_{1,n}=(M;S;g_{1/n})$ for each $n\in \BN\setminus \{0\}$ where
\be 
	\item $M$ is a saturated circle;
	\item $g_{1/n}$ is a rotation (clockwise) by $2\pi/n$-radian; and	
	\item $S$ is a ternary relation such that $S(a,b,c)$ holds if $a,b,c$ are distinct and $b$ comes before $c$ going around the circle clockwise starting at $a$.
	
\ee
and it was shown that the unique strong 1-type $p_n$ in $S_1(\emptyset)$ has the trivial first homology group for every $n$, which is actually a Lascar strong type. Here we consider two structure $\CM_1=(M;S;g_{1/n} : n\in \BN\setminus\{0\})$ expanding the structures $\CM_{1,n}$ by adding all rotation functions of $2\pi/n$-radian for each $n\in \BN\setminus\{0\}$ at the same time. When we write $g_r$ for $r=m/n$ in $\BQ\cap[0,1)$, it means $g_{1/n}^m$, and $\CM_2=(M;U_{<r},U_{=r}|r\in (0,1/2]\cap \BQ)$, where $U_{<r}(x,y)$ says the smallest length between $x$ and $y$ along the arc is less than $2\pi r$, and $U_{<r}(x,y)$ says the smallest length between $x$ and $y$ along the arc is exactly equal to $2\pi r$.

\subsection{Rosiness of $\Th(\CM_1)$ and $\Th(\CM_2)$}In this subsection, we mainly show that two theories of $\CM_1$ and $\CM_2$ are rosy. In \cite{EO}, C. Ealy and A. Onshuus gave a sufficient condition for being a rosy theory.

\begin{Fact}\label{character_rosy}
Any theory $T$ which geometrically eliminates imaginaries and for which algebraic closure defines a pregeometry is rosy of thorn $U$-rank $1$.
\end{Fact}
For rosiness of $\Th(\CM_i)$($i=1,2$), we show that $\Th(\CM_1)$ has weak elimination of imaginaries and $\Th(\CM_2)$ has geometric elimination of imaginaries. In \cite{P}, B. Poizat defined for a theory $T$ to have {\em weak elimination of imaginaries} if for every definable set has a smallest algebraically closed set which it is definable over. And $T$ has {\em a geometric elimination of imaginaries}, a weaker notion of weak elimination of imaginaries, if for each imaginary $e\in \CM^{eq}\models T$, there is a real tuple $\a\subset M$ such that $e\in\acl^{eq}(\a)$ and $\a\in \acl^{eq}(e)$. We give a sufficient condition of weak elimination of imaginaries for $\aleph_0$-categorical theory, used in \cite{KKL} :
\begin{theorem}\label{wei_omegacat}
Let $T$ be $\aleph_0$-categorical and let $\CM=(M,\ldots)$ be a saturated model of $T$. Suppose that for all $A$, $\acl(A)=\dcl(A)$. Suppose for a subset $X$ of $M^1$, if $X$ is $A_0(=\acl(A_0))$-definable and $A_1(=\acl(A_1))$-definable, then $X$ is $B(=A_0\cap A_1)$-definable. Then for a subset $Y$ of $M^n$, if $Y$ is $A_0$-definable and $A_1$-definable, then $Y$ is $B$-definable. Furthermore, in this case, $T$ has weak elimination of imaginaries.
\end{theorem}
\begin{proof}
Let $A_0=\acl(A_0)$, $A_1=\acl(A_1)$, and $B=A_0\cap A_1$. We use induction on $n$. If $n=1$, it holds by assumption. Let's show this holds for the case $n+1$ with inductive hypothesis for the case $n$. Let $A_0=\acl(A_0)$, $A_1=\acl(A_1)$, and $B=A_0\cap A_1$. We may assume $A_0$ and $A_1$ are finite, and so is $B$. Let $Y\subset M^{n+1}$ be $A_i$-definable, defined by formula $\phi_i(x_0,\ldots,x_n;\a_i)$ for $\a_i\subset A_i$ respectively. Then for each $c\in M$, the fiber of $Y$ over $c$, $Y_c:=\{ \x \in M^n|\ \phi_i(\x,c;\a)\}$ is $cB$-definable by induction. By $\aleph_0$-categoricity, there are only finitely many formulas over $\emptyset$ modulo $T$, and it easily follows that for each $y$, $\phi_i(x_0,\ldots,x_{n-1},y,\a_i)$ is $B$-definable. Thus $Y$ is $B$-definable.

And since there is no infinite descending chain of algebraically closed sets generating by finitely many elements, it makes for any definable set to have a smallest algebraically closed set where it is definable. Thus $T$ weakly eliminate imaginaries. 
\end{proof}
\noindent As a corollary of Theorem \ref{wei_omegacat}, we showed that for each $n\ge 2$, $\Th(\CM_{1,n})$ has weak elimination of imaginaries.
\begin{fact}\label{wei_M_1_reduct}\cite{KKL}
For each $n\ge 2$, $\Th(\CM_{1,n})$ weakly eliminates imaginaries.
\end{fact}

Next we will see that the theory of $\CM_1$ has quantifier-elimination.
\begin{definition}\label{rotation_closure}
Let $M$ be the underlying set of $\CM_1$ and $\CM_2$, so it is a saturated circle. For $A\subset M$, let $\cl(A):=\{g_r(a)|\ a\in A,\ r\in \BQ\cap [0,1) \}$. Later, we will see that $\cl(A)=\dcl_{\CM_1}(A)=\acl_{\CM_1}(A)$ in the home sort of $\CM_1$ and $\cl(A)=\acl_{\CM_2}$ in the home sort of $\CM_2$. It is also easy to see that $\cl(A)$ is a substructure of $\CM_1$ and $\CM_2$. 
\end{definition}

\begin{theorem}\label{QE_M_1}
The theory of $\CM_1$ has quantifier-elimination.
\end{theorem}
\begin{proof}
Take two small subset $A,B\subset M$ such that $A=\cl(A)$ and $B=\cl(B)$ in $M$. Take $a\in M\setminus A$. We will find $b\in M\setminus B$ such that the map $f\cup\{(a,b)\}$ is extended to an embedding from $\cl(Aa)$ to $\cl(Bb)$ in $\CM_1$. Then, the quantifier-elimination of $\Th(\CM_1)$ comes from a standard argument. We divide $A$ into two parts $A_0 :=\{x\in A|\ S(a,x,g_{1/2}(a))\}$ and $A_1 :=\{x\in A|\ S(g_{1/2}(a),x,a)\}$. Then $B$ is also divided into two parts $B_0=f(A_0)$ and $B_1=f(A_1)$. Take arbitrary $b\in M$ such that for all $y_0\in B_0,\ y_1\in B_1$, $S(y_1,b,y_0)$. Then $b$ is a desired element.
\end{proof}

\begin{theorem}\label{wei_M_1}
The theory of $\CM_1$ weakly eliminate imaginaries, and is rosy of thorn $U$-rank $1$.
\end{theorem}
\begin{proof}
In the structure $\CM_1$, there is no infinite descending chain of algebraical closure of finite sets by quantifier elimination. It is enough to show that if $X\subset M^n$ is $A_0(=\acl(A_0))$- and $A_1(=\acl(A_1))$-definable, then $X$ is $A_0\cap A_1(=B)$-definable. Then $X$ has a smallest algebraically closed set defining $X$, and $\Th(\CM_1)$ has weak elimination of imaginaries.

Let $A_i=\acl(A_i)=\cl(A_i)$ for $i=0,1$ and let $B=A_0\cap A_1$. Let $X\subset M^m$ be $A_i$-definable in $\CM_1$. Then $X$ is definable over $A_i$ for $i=0,1$ in some reduct $\CM_{1,n}$. Since $\CM_{1,n}$ weakly eliminates imaginaries, $X$ is definable over $B$ in $\CM_{1,n}$, defined by a formula $\psi(\x,\b)$. Then by the same formula $\psi(\x,\b)$, $X$ is $B$-definable in $\CM_1$.

By quantifier elimination, it is easily verified that the algebraic closure in $\CM_1$ gives a trivial pregeometry. Thus by Fact \ref{character_rosy}, $\Th(\CM_1)$ is a rosy theory having thorn $U$-rank $1$.
\end{proof}
\noindent There is only one $1$-strong type over empty set $p_0(x)=\{x=x\}$ in $\CM_1$.\\

Next we show that $\Th(\CM_2)$ has geometric elimination of imaginaries. We consider an expansion of $\CM_2$, $\CN_2=(M,U_{<r},U_{=r},g_{1/n})|r\in (0,1/2]\cap \BQ, n>0)$ by adding rotation functions $g_{1/n}$, and the reducts of $\CN_{2}$, $\CN_{2,k}=(M,U_{<r},U_{=r},g_{1/k}|\ r\in (0,1/2]_k)$, where $(0,1/2]_k=\{i/k|\ 0< i/k\le 1/2, i\in \BN\}$, for each $k\ge 1$.
\begin{theorem}\label{wei_N_2}
\be
	\item For each $k\ge 3$, $\Th(\CN_{2,k})$ has quantifier-elimination, and weak elimination of imaginaries, and it is $\omega$-categorical.
	\item $\Th(\CN)$ has weak elimination of imaginaries.
\ee
\end{theorem}
\begin{proof}
(1) Fix $k\ge 3$. Each binary relations $U_{<r}$ and $U_{=r}$ for $r\in (0,1/2]_k$ are $\emptyset$-definable in $\CM_{1,k}$ by quantifier-free formulas. So, by Fact \ref{wei_M_1_reduct}, it is enough to show that the ternary relation $S(x,y,z)$ is $\emptyset$-definable in $\CN_{2,k}$ by a quantifier-free formula. Denote $g_{i/k}(x)<y<g_{(i+1)/k}(x)$ for the formula $U_{<1/k}(g_{i/k}(x),y)\wedge U_{<1/k}(g_{(i+1)/k}(x),y))$. Consider the following quantifier-free formula 
$$
\begin{array}{c c l}
S_k'(x,y,z)&\equiv&\bigvee_{i=1}^{k-1}\limits [ z=g_{i/k}(x)\rightarrow (\bigvee_{j=1}^{i-1}\limits y=g_{j/k}(x)\\
&\vee& \bigvee_{j=0}^{i-1}\limits g_{j/k}(x)<y<g_{(j+1)/k}(x))]\\
&\vee& \bigvee_{i=0}^{k-1}\limits[ g_{i/k}(x)<z<g_{(i+1)/k}(x)\rightarrow\\
&&( g_{i/k}(x)<y<g_{(i+1)/k}(x)\wedge g_{-1/k(z)}(z)<y<z)\\
&\vee&\bigvee_{j=0}^{i-1}\limits g_{j/k}(x)<y<g_{(j+1)/k}(x)],
\end{array}
$$
and this formula defines $S(x,y,z)$ in $\CN_{2,k}$.\\

(2) Consider the ternary relation $S'_k$ for some $k\ge 3$. Then $S'_k$ also defines the ternary relation $S$ in $\CN$. Thus, as the relation between $\CN_{2,k}$ and $\CM_{1,k}$, by Theorems \ref{QE_M_1} and \ref{wei_M_1}, the theory of $\CN$ has quantifier-elimination and weakly eliminate imaginaries.
\end{proof}

\begin{theorem}\label{wei_M_2}
The theory of $\CM_2$ geometrically eliminate imaginaries.
\end{theorem}
\begin{proof}
At first, we define some equivalence relations $E_n$ on $M^n$, which are $\emptyset$-definable in $\CM_2$. Let $n\ge 1$. Define a formula $$C_n(z_1,\ldots,z_n)\equiv \bigwedge_{1\le i<j\le n}\limits z_i\neq z_j\wedge \bigwedge_{1\le i<j\le n}\limits \neg U_{<1/n}(z_i,z_{j})$$ so that it defines the set $\{ z,g_{1/n}(z),\ldots, g_{(n-1)/n}(z) \}$. Next, we define the following formula $$\begin{array}{c c l}
I_n(x,y;z_1,\ldots,z_n)&\equiv&C_n(\z)\\
&\wedge&[\bigvee_{i=1}^{n}\limits (y=z_i)\rightarrow(U_{<1/n}(z_i,x)\wedge U_{<1/n}(z_{i+1},x))]\\
&\vee&[\bigvee_{i=1}^{n}\limits (U_{<1/n}(z_i,x)\wedge U_{<1/n}(z_{i+1},x))\rightarrow\\
&&\exists y_2\cdots y_n(C_n(y,y_2,\ldots,y_n)\\
&\wedge&(U_{<1/n}(z_{i+1},y_2)\wedge U_{<1/n}(z_{i+2},y_2))\\
&\wedge&(U_{<1/n}(y,x)\wedge U_{<1/n}(y_2,x))]
\end{array}
$$
,where $z_{n+1}=z_1$ and $z_{n+2}=z_2$, which defines the sets $\{(a,b)|\ b<a<g_{1/n}(b)\}$ or $\{(a,b)|\ g_{-1/n}(b)<a<b\}$ up to $z_i=g_{(i-1)/n}(z_1)$ or $z_i=g_{-(i-1)/n}(z_1)$ respectively. At last, consider the following equivalence relation $E_n$ on $M^n$ defined as follows $$E_n(\z,\z')\equiv \forall xy(I_n(x,y;\z)\leftrightarrow I_n(x,y;\z')).$$ Then $|M^n/E_n|=3$ and each classes represents one of the following tuples :
$$
\begin{cases}
(z,g_{1/n}(z),\ldots,g_{(n-1)/n}(z))\\
(z,g_{-1/n}(z),\ldots,g_{-(n-1)/n}(z))\\
\mbox{other wise}
\end{cases}
$$
Let $\CL(\CM_2)=\{U_{<r},U_{=r}\}$ be the language of $\CM_2$ and $\CL(\CN_2)=\CL(\CM_2)\cup\{g_{1/n}\}_{n\ge 1}$ be the language of $\CN_2$. Expand $M_2$ to $M_2'=(M,U_{<r},U_{=r},a_n,b_n,c_n)_{n\ge 1}$ by adding imaginary elements such that $\{a_n,b_n,c_n\}=M^n/E_n$, where $a_n=[z,g_{1/n}(z),\ldots,g_{(n-1)/n}(z)]_{E_n}$, and $b_n=[z,g_{-1/n}(z),\ldots,g_{-(n-1)/n}(z)]_{E_n}$, for each $n\ge 1$, and the language of $\CM_2$, $\CL(\CM_2)$ is the union of $\CL(\CM_2)$ and $\{S_{E_n},f_{E_n}\}_{n\ge 1}$, where $S_{E_n}$ is interpreted as the sort for $M^n/E_n$ and $f_{E_n}$ is as the canonical function from $M_n$ into $S_{E_n}$ such that for $\a,\b\in M^n$, $f_{E_n}(\a)=f_{E_n}(\b)$ if and only if $E_n(\a,\b)$.

\begin{claim}\label{g_n_definalbe_M_2'}
Each function $g_{1/n}$ is definable in $M_2'$.
\end{claim}
\begin{proof}
Two functions $g_1$ and $g_{1/2}$ are already definable in $M_2$ and so in $M_2'$ also. Let $n\ge 3$. For each $a\in M$, there is only one element $a'$ in $M$ such that $\exists x_3,\ldots,x_n (C_n(a,a',x_3,\ldots,x_n)\wedge f_{E_n}(a,a',x_3,\ldots,x_n)=a_n)$, and thus $a'=g_{1/n}(a)$. Therefore, the graph of $g_{1/n}$ is defined by the formula $\exists x_3,\ldots,x_n (C_n(x,y,x_3,\ldots,x_n)\wedge f_{E_n}(x,y,x_3,\ldots,x_n)=a_n)$. 
\end{proof}

\begin{claim}\label{equiv_M_2'eq_N_2eq}
$(\CM_2')^{eq}=\CN_2^{eq}$.
\end{claim}
\begin{proof}
For each $e\in (\CM_2')^{eq}$, since $M^n/E_n\subset \CM_2^{eq}$, $e$ is in $(\CM_2^{eq})^{eq}=\CM_2^{eq}$. By the way, $\CM_2$ is a reduct of $\CN_2$, and $e\in \CN_2^{eq}$. Conversely, each $e'\in \CN_2^{eq}$ is clearly in $(\CN_2')^{eq}$, wehre $\CN_2'=(M,U_{<r},U_{=r},g_{1/n},a_n,b_n,c_n)$. By Claim \ref{g_n_definalbe_M_2'}, $(\CN_2')^{eq}=(\CM_2')^{eq}$ and $e'\in (\CM_2')^{eq}$. 
\end{proof}
Take $e\in \CM_2^{eq}$ arbitrary. Since $\CM_2$ is a reduct of $\CN_2$, $e$ is in $N_2^{eq}$. By weak elimination of imaginaries of $\Th(\CN_2)$, there is a finite tuple $\b\subset M$ such that $$e\in \dcl_{\CN_2}^{eq}(\b)\mbox{ and } \b\in\acl_{\CN_2}^{eq}(e).$$ By Claim \ref{equiv_M_2'eq_N_2eq}, $$e\in \dcl_{\CM_2'}^{eq}(\b)\mbox{ and } \b\in\acl_{\CM_2'}^{eq}(e).$$ Since $M^n/E_n\subset \acl_{\CM_2}^{eq}(\emptyset)$, $$e\in \acl_{\CM_2}^{eq}(\b)\mbox{ and } \b\in\acl_{\CM_2}^{eq}(e).$$ Therefore, each imaginaries in $\CM_2^{eq}$ is inter-algebraic with a finite tuple in the home sortin $\CM_2$ and $\Th(\CM_2)$ has geometric elimination of imaginaries.
\end{proof}
\noindent Now we show that $\Th(\CM_2)$ is a rosy theory having thorn-$U$ rank $1$.
\begin{theorem}
The theory of $\CM_2$ is a rosy theory of thorn-$U$ rank $1$.
\end{theorem}
\begin{proof}
By Fact \ref{character_rosy} and Theorem \ref{wei_M_2}, it is enough to show that the algebraic closure in the home sort gives a trivial pregeometry in $\CM_2$. For any $A\subset M$, it is clear that $\cl(A)\subset \acl_{\CM_2}(A)$. By the way, from Theorem \ref{wei_N_2}, $\acl_{\CN_2}(A)=\cl(A)$. Since $\CM_2$ is a reduct of $\CN_2$, $\acl_{\CM_2}(A)\subset \acl_{\CN_2}(A)$, and thus $\cl(A)=\acl_{\CM_2}(A)$. So, the algebraic closure in $\CM_2$ gives a trivial pregeometry.
\end{proof}
\noindent The theory of $\CM_2$ does not eliminate quantifier but there is only one $1$-type over $\acl^{eq}(\emptyset)(\neq \emptyset)$, $q_0(x)=\{x=x\}$.
\subsection{Computation of $H_1$ in $\CM_1$}
In the section 2, we see that the first homology groups are determined by end-point pairs of $1$-shells. In $\CM_1$, for a fixed $a\in M$, we observe that $S_1(a)$ looks like a circle with a rotation. From this observation, we compute the first homology group of $p_0$ in $\CM_1$ :
\begin{theorem}\label{h1_in_M_1}
In $\CM_1$, the first homology group of $p_0$ is isomorphic to $\BR/\BZ$.
\end{theorem}
We start with defining a distance-like notion between two points on $M$. For a subset $A$ in $\BR$, we denote $A_{\BQ}$ for $A\cap \BQ$. 
\begin{definition}\label{distance}
Let $a,b\in M$ be two elements. We define the \em{$S$-distance} of $b$ from $a$, denoted by $\sdist(a,b)$ as follows : For $r\in \BQ$ and $s<t\in [0,1)_{\BQ}$,

\be 
	\item $\sdist(a,b)=r$ if $b=g_r(a)$;
	\item $s<\sdist(a,b)<1$ if $\CM_1 \models S(g_s(a),b,a)$;	
	\item $0<\sdist(a,b)<t$ if $\CM_1 \models S(a,b,g_t(a))$; and
	\item $s<\sdist(a,b)<t$ if $\CM_1 \models S(g_s(a),b,g_r(a))$.
	
\ee

\noindent For $r\in [0,1)\setminus \BQ$, we write $\sdist(a,b)=r$ if for $s<t\in [0,1)_{\BQ}$ and $s<r<t$, $s<\sdist(a,b)<t$. Let $r \in (0,1)_{\BQ}$. We write $\sdist(a,b)=r-\epsilon$ if for $s\in (0,1)_{\BQ}$ with $s<r$, $s<\sdist(a,b)<r$. We write $\sdist(a,b)=r+\epsilon$ if for $t\in (0,1)_{\BQ}$ with $r< t$, $r<\sdist(a,b)<t$. We write $\sdist(a,b)=\epsilon$ if for all $s\in (0,1)_{\BQ}$, $0<\sdist(a,b)<s$. We write $\sdist(a,b)=1-\epsilon$ if for all $s\in (0,1)_{\BQ}$, $s<\sdist(a,b)<1$.
\end{definition}
\noindent For a subset $A\subset \BQ$, we define $A^*:=A\cup\{x\pm\epsilon| x\in A \}$. This $S$-distance has the values in $[0,1)\cup[0,1)^*_{\BQ}\cup\{1-\epsilon\}$. In Appendix B, using Dedekind cut, we develop multivalued operations $+^*,\times^*,-^*$ to make $\BR\cup\BQ^*$ a ring-like structure.  Now we extend the values of $S$-distance to $\BR\cup\BQ^*$. Since $g_k=\id$ for all $k\in \BZ$, we write $\sdist(a,b)=r$ for $r\in \BR\cup\BQ^*$ if $\sdist(a,b)=r'$ where $r'$ is the unique number in $[0,1)\cup [0,1)_{\BQ}^*$ such that $r\in r'+^*n$ for some $n\in \BZ$. Then this values depends only on the type of $(a,b)$, that is, for $a_0,a_1,b_0,b_1\in M$, if $a_0b_0\equiv a_1b_1$, then $\sdist(a_0,b_0)=\sdist(a_1,b_1)$ of the values in $[0,1)\cup[0,1)_{\BQ}^*$. Then the following fact is easily verified :
\begin{fact}\label{direct-distance}
Let $a,b,c\in M$.
\be 
	\item $\sdist(b,a)=1-^*\bsdist(a,b)$.
	\item $\sdist(a,c)=\bsdist(a,b)+^*\bsdist(b,c)$ modulo $\BZ\cup \BZ^*$, that is, $\sdist(a,b)+^*\sdist(b,c)-^*\sdist(a,c)\subset \BZ^* $.
\ee
\noindent By (1), $\sdist$ is not symmetric, that is, for $a,b\in M$, $\sdist(a,b)\neq\sdist(b,a)$ and so it is called a directed distance. 
\end{fact}

\noindent Now we assign each $1$-simplex $f$ a value $n_f$ in $\BR\cup\BQ^*$ as follows : There are $a,b\in M$ such that $[a,b]=f$, and we define $n_f$ as $\sdist(a,b)$. Then $n_f$ is well-define, that is, it does not depend on the choice of $a,b$ because if $a_i,b_i\in M$ satisfy $[a_0,b_0]=[a_1,b_1]=f$, then $a_0 b_0\equiv a_1 b_1$ and $\sdist(a_0,b_0)=\sdist(a_1,b_1)$. We also assign each $1$-shell $s=f_{01}+f_{12}-f_{02}$ to a multivalue $n_s$ in $\BR\cup\BQ^*$ as follows : $n_s=n_{f_{01}}+^*n_{f_{12}}-^* n_{f_{02}}$. This value is also related with the distance of end points. Let $(a,a')$ be an endpoint pair of $s$, then $\sdist(a,a')=n_s$ modulo $\BZ^*$. Using this assignment of 1-shells, we give a necessary and sufficient condition for a $1$-shell to be a boundary of a $2$-chain :

\begin{theorem}\label{NSfor[s]=0inM}
A 1-shell $s=f_{12}-f_{02}+f_{01}$ is a boundary of a 2-chain in $p$ if and only if 
\begin{center}
$n_s=n_{01}+^*n_{12}+^*n_{20}\subset\BZ^*$,
\end{center}
where $n_{01}=n_{f_{01}},\ n_{12}=n_{f_{12}},\ n_{20}=-^* n_{f_{02}}$. Moreover it is equivalent to that the two end points of $s$ are Lascar equivalent over $\emptyset$. 
\end{theorem}
\begin{proof}
($\Rightarrow$) Let $\alpha$ be a 2-chain having the boundary $s$. By Fact \ref{supp=3}, we may assume $\alpha=\sum_{i=0}^{2n}\limits (-1)^i a_i$ be a chain-walk from $f_{01}$ to $-f_{02}$ with $\supp(\alpha)=\{0,1,2\}$. Let $[3]=\{0,1,2\}$. From Theorem \ref{chainwalk_representation}, there are independent elements $d_0,d_1,\cdots,d_{2n+2}$ such that 

\begin{itemize}
	\item $a_i([3])\equiv [d_0,d_{2i+1},d_{2i+2}]_{[3]}$ if $i$ is even, and $a_i([3])=[d_0,d_{2i+2},d_{2i+1}]_{[3]}$ if $i$ is odd;
	\item For some even number $0 \le i_0\le 2n$, $[d_{2i_0+1},d_{2i_0+2}]_{\{1,2\}}=f_{12}(\{1,2,\})$; and
	\item For each even number $0 \le j_o\neq i_0 \le 2n$, there is an odd number $0\le j_1\le 2n$ such that $[d_{2j_0+1},d_{2j_0+2}]_{\{1,2\}}=[d_{2j_1+2},d_{2j_1+1}]_{\{1,2\}}$.
\end{itemize}

\noindent Then $\sdist(d_1,d_0)+^* \sdist(d_0,d_{2n+2})=-^* n_{01}-^* n_{20}$ and by Fact \ref{direct-distance} (1), $\sdist(d_1,d_2)+^* \sdist(d_2,d_3)+^*\cdots+^* \sdist(d_{2n+1},d_{2n+2})=n+^* n_{12}$. By Fact \ref{direct-distance} (2), $\sdist(d_1,d_{dn+2})\in (-^* n_{01}-^* n_{20})\cap (n+^* n_{12})$. Since $\{0\}^*=(-^* n_{01}-^* n_{20})+^*(n_{01}+^*n_{20})$ and $(n+^* n_{12})+^*(n_{01}+^*n_{20})=n+^*(n_{01}+^*n_{12}+^*n_{20})$, these two equations imply $n+^*(n_{01}+^*n_{12}+^*n_{20})\subset \{0\}^*$. Therefore, $n_{01}+^*n_{12}+^*n_{20}\subset \{0\}^*-^* \{n\}^*=\{-n\}^*$ for $n\in \BN\subset \BZ$.\\

($\Leftarrow$) Suppose $n_{01}+^*n_{12}+^*n_{20}\subset\{n\}^*$ for some $n\in \BZ$. Then there are independent elements $a,b,c,a'$ such that
$$[a,b]_{\{0,1\}}=f_{01}(\{0,1\}),\ [b,c]_{\{1,2\}}=f_{12}(\{1,2\}),\ [a',c]_{\{0,2\}}=f_{02}(\{0,2\}).$$
So, $\sdist(a,b)=n_{01},\ \sdist(b,c)=n_{12},\ \sdist(c,a)=n_{20},$ and $\sdist(a,a')\in n_{01}+^*n_{12}+^*n_{20}$. Thus $\sdist(a,a')\in \{n\}^*$ and $\sdist(a,a')\in \{0\}^*$. Since $\{a,a'\}$ is independent, $\sdist(a,a')\in \{0\}^*\setminus \{0\}$, that is, $\sdist(a,a')=\epsilon$ or $\sdist(a,a')=1-\epsilon$.

We will find $d\in M$ such that $a\equiv_d a'$ and $d\indo abca'$. Consdier a partial type $\Sigma(x)=\{s<\sdist(x,a)<t\leftrightarrow s<\sdist(x,a')<t\}_{s<t\in [0,1]}$. Consider finitely many pairs $(s_i,t_i)$ with $s_i<t_i$ and a formula $$\bigwedge (s_i<\sdist(x,a)<t_i\leftrightarrow s_i<\sdist(x,a')<t_i).$$
We may assume $s_i \le s_0 < t_0 \le t_i$. It is enough to show that
$$s_0<\sdist(x,a)<t_0\leftrightarrow s_0<\sdist(x,a')<t_0$$
is satisfiable. Suppose $s_0<\sdist(x,a)<t_0$ is satisfiable. Then there is a pair $(s,t)$ such that $s_0<s<t<t_0$ and $s<\sdist(x,a)<t$ is satisfiable. Let $e\in M$ be independent from $a$ such that $s<\sdist(e,a)<t$ holds. Since $\sdist(a,a')\in \{0\}^*\setminus \{0\}$, there is is a pair $(s',t')$ such that $s'<t'$, $s_0<s+s'<t+t'<t_0$, and $s'<\sdist(a,a')<t'$. Then, $s<\sdist(e,a)<t$ and $s'<\sdist(a,a')<t'$ imply $s+s'<\sdist(e,a')<t+t'$. Since
$s_0<s+s'<t+t'<t_0$, $s_0<\sdist(e,a')<t_0$ and $s_0<\sdist(x,a')<t_0$ is satisfiable. By the same way, $s_0<\sdist(x,a')<t_0\rightarrow s_0<\sdist(x,a)<t_0$.

Therefore, there is $d\in M$ such that $\Sigma(d)$ and $\sdist(d,a)=\sdist(d,a')$. Moreover we may assume that $\{a,b,c,a',d\}$ is independent by taking $d\indo_{aa'}bc$. Then, there is a 2-chain $\alpha=a_0+a_1-a_2$, where
\begin{itemize}
	\item $\supp(a_0)=\{0,1,3\}$, $\supp(a_1)=\{1,2,3\}$, and $\supp(a_2)=\{0,2,3\}$;
	\item $a_0(\{0,1,3\})=[a ,b,d]_{\{0,1,3\}}$, $a_1(\{1,2,3\})=[b,c,d]_{\{1,2,3\}}$, and $a_2(\{0,2,3\})=[a' ,c,d]_{\{0,2,3\}}$;
	\item $a_0^{01}=f_{01}$, $a_1^{12}=f_{12}$, and $a_2^{02}=f_{02}$; and
	\item $a_0^{03}=a_2^{03}$, $a_0^{13}=a_1^{13}$, and $a_1^{23}=a_2^{23}$.
\end{itemize}

\noindent Then $\Bd\alpha=f_{0,1}+f_{1,2}-f_{02}+(a_2^{03}-a_0^{03})$ and $\Bd \alpha=f_{0,1}+f_{1,2}-f_{02}$.\\

Now we show moreover part. Let $a,a'$ be end points of $s$. If $a\equiv^L a'$, then $s$ is a boundary of $2$-chain and $n_s\subset \{n\}^*$ for some $n\in\BZ$. Conversely, we assume that $n_s\subset \{n\}^*$ for some $n\in \BZ$. In the proof of right-to-left, we found $d\in M$ such that $a\equiv_d a'$. Consider a substructure generated by $d$. $\cl(d)=\dcl(c)=\acl(c)$. Then $a\equiv_{cl(d)} a'$ and $a\equiv^L a'$.
\end{proof}
Now we are ready to prove Theorem \ref{h1_in_M_1}. Define a map $\Phi:\ H_1(p_0)\rightarrow (\BR\cup\BQ^*)/\BZ^*$ as sending $[s]$ into $n_s+^*\BZ^*$, and $(\BR\cup\BQ^*)/\BZ^*\cong \BR/\BZ$ in Appendix B. It is easy to see this map is surjective. Since for an endpoint pair $(a,b)$ of $s$, $n_s+^*\BZ^*=\sdist(a,b)+^*\BZ^*$, this map $\Phi$ depends on the endpoint pairs of $1$-shells. By Theorem \ref{endpt_gpstr}, given $1$-shells $s_0$ and $s_1$ and endpoint pairs $(a,b)$ and $(b,c)$ with respect to $s_0$ and $s_1$, there is an $1$-shell $s$ such that $[s]=[s_0]+[s_1]$ and $(a,c)$ is an endpoint pair of $s$, thus this map $\Phi$ is a group homomorphism. Moreover by Theorem \ref{NSfor[s]=0inM}, it is injective, and therefore it is an isomorphism.

Therefore, we show that the first homology group in $p_0$ is isomorphic to $\BR/\BZ$, and it is interesting that from Theorem \ref{NSfor[s]=0inM}, this first homology group is exactly isomorphic to the Lascar group $\gal_L (\CM_1;\emptyset)$.
\subsection{Computation of $H_1$ in $\CM_2$}
In this subsection, we compute the first homology group of $q_0$. Since $q_0$ is over $\acl^{eq}(\emptyset)$, we work in $\CM_2^{eq}$ with constant elements in $\acl^{eq}(\emptyset)$. Since each elements in $M^n/E_n$ is already in $\acl^{eq}(\emptyset)$, and by Theorem \ref{wei_N_2}, we may work in $(\CM_2')^{eq}=\CN_2^{eq}$. But already noted in the proof of \ref{wei_N_2} (2), the ternary relation $S(x,y,z)$ is definable in $\CN_2$ and thus we work in $\CM_1^{eq}$. So, by the previous subsection, the first homology group of $q_0$ is same with one of $p_0$ and it is isomorphic to $\BR/\BZ$, which is also Lascar group over $\acl^{eq}(\emptyset)$ in $\CM_2$.\\

\medskip

We conjecture that there are only automorphisms described in Theorem \ref{kernel_canonical_epi} in the kernel of the canonical epimorphism in Theorem \ref{canonical_epi}. 
\begin{question}\label{qeustion_kernel}
Let $T=T^{eq}$ be a rosy theory, and $\CC\models T$. For $A=\acl(A)$, for a strong type $p$ over $A$, let $\Psi:\ \aut_A(\CC)\rightarrow H_1(p)$ be a canonical epimorphism. Then, is the kernel of $\Psi$ exactly generated by automorphisms in the following :
\begin{enumerate}
\item $\aut_{\acl(aA)}(\CC)$ for $a\models p$;
\item $\autf(\CC)$; and
\item $(\Psi')^{-1}([G,G])$, where $G$ and $\Psi'$ are in Theorem \ref{kernel_canonical_epi},
\end{enumerate}
so, is $H_1(p)\cong (\gal_L(\CC;A)/<\bar{\sigma}:\ \sigma\in \aut_{\acl(Aa)},\ a\models p>)^{ab}$? Furthermore if any algebraic closure is again a substructure, then is $H_1(p)\cong \gal_L(\CC;A)^{ab}$? 
\end{question}
\noindent Fortunately, the answer for Question \ref{qeustion_kernel} is yes for known examples in \cite{GKK}\cite{KKL} and our two examples.

\section{Appendix}
\subsection{Appendix A} We show the possible number of bounded type-definable equivalence classes on a strong type is $1$ or at least $2^{\aleph_0}$. Let $T(=T^{eq})$ be any theory of a language $\CL$ and let $\CC$ be a monster model of $T$. Fix a small subset, $A=\acl(A)$ and choose a strong type $p(\x)$ over $A$(with $\x$ of possibly infinite length). We shall denote $\x$ as $x$ conventionally.
\begin{theorem}
Let $E(x,y)$ be a bounded $A$-type-definable equivalence relation on $p(x)$ and denote $p/E$ for the set of $E$-classes on $p$. Then, \begin{center}
$|p/E|=1$ or $\ge 2^{\aleph_0}$.
\end{center}
\end{theorem}
\begin{proof}
For a convention, we assume $A=\emptyset$. We divide two cases that $p/E$ is finite and $p/E$ is infinite.\\

Case 1. $p/E$ is finite : Suppose $p/E$ is finite. Let $a_0,\cdots, a_n\models p$ be representatives of all distinct classes in $p/E$, and let $\a=(a_0,a_1,\ldots ,a_n)$. At first, we show that $E$ is relatively definable on $p$. Consider two type-definable formula $E(x,a_0)$ and $\bigvee_{i>0}E(x,a_i)$ partitioning $p$, and by compactness, $p(x)\models E(x,a_0)\leftrightarrow \phi(x,a_0)$ for some formula $\phi(x,z)$ such that $E(x,a_0)\models \phi(x;a_0)$. Since $a_0\equiv a_i$, $p(x)\models E(x,a_i)\leftrightarrow \phi(x,a_i)$ for all $i\le n$. Thus, $p(x)\wedge p(y)\models E(x,y)\leftrightarrow \psi(x,y;\a)$, where $\psi(x,y;\z)=\bigvee_i [\phi(x,z_i)\wedge\bigvee_{j\neq i}\phi(y,z_j )]$. Since $E$ is invariant, $p(x)\wedge p(y)\wedge \psi(x,y;\z)\wedge \tp(\a)(\z)\models \psi(x,y;\a)(\leftrightarrow E(x,y))$. By compactness, there is a formula $\psi'(\z)$ in $\tp(\a)(\z)$ such that $p(x)\wedge p(y)\wedge \psi(x,y;\z)\wedge \psi'(\z)\models \psi(x,y;\a)$. Take $\theta(x,y)\equiv \exists \z (\psi'(\z)\wedge \psi(x,y;\z))$. Then $p(x)\wedge p(y)\models \theta(x,y)\leftrightarrow \psi(x,y;\a)$. Therefore $E$ is relatively definable on $p$ by the formula $\theta$. Moreover, we may assume $\theta(x,y)$ is a reflexive and symmetric relation by taking $x=y\vee (\theta(x,y)\wedge \theta(y,x))$. 


Next, we find a finite $\emptyset$-definable equivalence relation $E'$ such that $p(x)\wedge p(y)\models E(x,y)\leftrightarrow E'(x,y)$. Since $E$ is an equivalence relation, 
$$\begin{array}{c c l}
p(x)\wedge p(y)\wedge p(z)&\models& \bigvee_i\limits \theta(x,a_i)\wedge\bigvee_i\limits \theta(y,a_i)\wedge\bigvee_i\limits \theta(z,a_i)\\
&\wedge& \bigwedge_i \limits (\theta(x,a_i)\rightarrow \bigwedge_{i\neq j}\limits \neg \theta(x,a_j))\\
&\wedge&\bigwedge_i \limits (\theta(y,a_i)\rightarrow \bigwedge_{i\neq j}\limits \neg \theta(y,a_j))\\
&\wedge&\bigwedge_i \limits (\theta(z,a_i)\rightarrow \bigwedge_{i\neq j}\limits \neg \theta(z,a_j)) \\
&\wedge&(\theta(x,y)\wedge\theta(y,z)\rightarrow \theta(x,z)).\ (*)
\end{array}$$
Again by compactness, there is $\delta(x)\in p(x)$ such that $$\delta(x)\wedge\delta(y)\wedge\delta(z)\models (*).$$ Define a definable equivalence relation $E'(x,y)\equiv [\neg \delta(x)\wedge \neg \delta(y)]\vee [\delta(x)\wedge \delta(y)\wedge \forall z(\delta(z)\rightarrow (\theta(z,x)\leftrightarrow \theta(z,y)))]$.
\begin{claim}\label{fin_equiv_relation}
The equivalence relation $E'$ is finite.
\end{claim}
\begin{proof}
First, $\neg\delta(x)$ is a $E'$-class. We show that on $\delta$, the $E'$-classes are of the form of $\theta(x,a_i)\wedge \delta(x)$. By the choice of $\delta$, it is partitioned by $\{\theta(x,a_i)\wedge \delta(x)\}_{i\le n}$.\\

1) We show that $\models \theta(x,a_i)\wedge \delta(x)\rightarrow E'(x,a_i)$ : Choose $b\models \theta(x,a_i)\wedge \delta(x)$. Take $c\models \delta(x)\wedge \theta(x,a_i)$. Since $\theta$ is transitive on $\delta$ and $\theta(b,a_i)$ holds, $\theta(c,b)$ holds. Conversely, if $d\models \delta(x)\wedge \theta(x,b)$, then by transitivity of $\theta$ on $\delta$, $\theta(d,a_i)$ holds. Therefore, $E'(b,a_i)$ holds.\\

2) For $i\neq j$, $\neg E'(a_i,a_j)$  : Suppose for some $i\neq j$, $E'(a_i,a_j)$ holds. Then $\theta(a_i,a_j)$ holds but it is impossible since $a_i,a_j\models p$ and $\theta$ and $E$ are same on $p\times p$.\\

\noindent By 1) and 2), the $E'$-classes are of the form of $\theta(x,a_i)\wedge \delta(x)$ or $\neg\delta(x)$ and $E'$ is a finite equivalence relation.
\end{proof}

\noindent From the proof of Claim \ref{fin_equiv_relation}, $E'$ and $E$ are the same equivalence relation on $p\times p$. Since $E'$ is finite and $p$ is a strong type, $p/E=p/E'$ and there are only one $E$-class in $p$.\\

Case 2. $p/E$ is infinite : Suppose $p/E$ is infinite. Let $\kappa=|p/E|$. If $E$ is definable, then by compactness, $|p/E|\ge \kappa'$ for any $\kappa'<|\CC|$ and $E$ is not bounded. So $E$ is type-definable and $E(x,y)\equiv \bigwedge_{i<\lambda}\limits \phi_i(x,y)$, where $\phi_i(x,y)$ is a formula and $\lambda$ is an infinite cardinal. We may assume $\phi_i(x,y)$ is reflexive and symmetric by taking $x=y\vee(\phi_i(x,y)\wedge \phi_i(y,x))$ instead of $\phi_i(x,y)$ for each $i<\lambda$, and that for $i<j<\lambda$, $\models \phi_j(x,y)\rightarrow\phi_i(x,y)$ $(\dagger)$ by taking $\phi_j(x,y)\wedge\phi_i(x,y)$. Moreover, by compactness, we may assume that for $\models \exists z(\phi_{i+1}(x,z)\wedge \phi_{i+1}(z,y))\rightarrow \phi_i(x,y)$ $(\ddagger)$. Let $\{a_k\models p\}_{k<\kappa}$ be the set of representatives of $E$-classes.

\begin{claim}\label{inf_classes_in_phi}
For each $i<\lambda$ and $k<\kappa$, $\phi_i(x,a_k)(\CC)$ contains infinitely many $E$-classes.
\end{claim}
\begin{proof}
Fix $i<\lambda$. By compactness, there are finitely many $k_0<k_1<\cdots<k_n$ such that $p\models \bigvee_j \phi_i(x,a_{k_j})$. By Pigeonhole Principle, some $\phi_i(x,a_{k_l})$ contains infinitely many $a_k$'s. By $(\dagger)$ and $(\ddagger)$, $\phi_i(x,a_{k_l})$ contains infinitely many $E$-classes. Since $a_n\equiv a_m$ for $n,m<\kappa$ and $E$ is invariant, each $\phi_i(x,a_k)$ contains infinitely many $E$-classes.
\end{proof}

\begin{claim}\label{disj_phi_in_phi}
For each $i<\lambda$ and $k<\kappa$, there are $i<j<\lambda$ and $k_0,k_1<\kappa$ such that $$\models [(\phi_j(x,a_{k_0})\vee\phi_j(x,a_{k_0}))\rightarrow \phi_i(x,a_k)]\wedge [\neg \exists x (\phi_j(x,a_{k_0})\wedge \phi_j(x,a_{k_1}))].$$
\end{claim}
\begin{proof}
Fix $i<\lambda$ and $k<\kappa$. By Claim \ref{inf_classes_in_phi}, $\phi_i(x,a_k)$ contains infinitely many $E$-classes. Choose two $E$-classes in $\phi_i(x,a_k)$ and let $a_{k_0}$ and $a_{k_1}$ be representatives of two classes respectively. Since $E(x,a_{k_0})(\CC)$ and $E(x,a_{k_1})(\CC)$ are disjoint, by compactness, for some $j>i$, $\phi_j(x,a_{k_0})(\CC)$ and $\phi_j(x,a_{k_1})(\CC)$ are disjoint and we are done.
\end{proof}
\noindent From Claim \ref{inf_classes_in_phi}, \ref{disj_phi_in_phi} and the fact that the cofinality of $\lambda$ is at least $\aleph_0$, we get a binary tree $\CB :\ 2^{< \omega} \rightarrow \omega\times \kappa$ such that for each $b\in 2^{<\omega}$, $\CB(\stackrel\frown{b0})=(j,k_0)$ and $\CB(\stackrel\frown{b1})=(j,k_1)$ where if $\CB(b)=(i,k)$, then $j<\omega$ and $k_0,k_1<\kappa$ satisfies Claim \ref{disj_phi_in_phi} for $(i,k)$. Then for each $\tau \in 2^{\omega}$, we get a set of formula $\{\phi_{i(\tau\upharpoonright n )}(x,a_{k(\tau\upharpoonright n)})\}$, where $\CB(\tau\upharpoonright n)=(i(\tau\upharpoonright n ),k(\tau\upharpoonright n))$ for each $n\in \omega$. By the choice of $\CB$, for $\tau_0\neq \tau_1 \in 2^{\omega}$, $\bigcap_n \phi_{i(\tau_0\upharpoonright n )}(x,a_{k(\tau_0 \upharpoonright n)})(\CC)$ and $\bigcap_n \phi_{i(\tau_1 \upharpoonright n )}(x,a_{k(\tau_1\upharpoonright n)})(\CC)$ are disjoint, and each contains $E$-classes. Thus, $p/E$ has at least $2^{\aleph_0}$ many elements.
\end{proof}

\subsection{Appendix B}
We see how to recover a real ordered group $(\BR,+)$ from a dense linear order extending $(\BQ,<)$ using Dedekind cut. Consider a language $\CL_{od,\BQ}=\{<\}\cup\{r\}_{r\in\BQ}$ and a $\CL_{od,\BQ}$-structure $\CU=(U,<,r : r\in \BQ)$ which is a saturated dense linear order extending $(\BQ,<)$. Then $\Th(\CU)$ has quantifier elimination.

Consider the $1$-types over empty set, $S_1(\emptyset)(=S_1)$. Then by quantifier elimination, any $1$-type $p$ has one of the following forms : For $r\in \BQ$ and $r'\in \BR\sm\BQ$,
\be 
	\item $\{x=r\}$;
	\item $\{l<x<r|\ l<r\}$;
	\item $\{r<x<u|\ r<u\}$; and
	\item $\{l<x<u|\ l<r'<u\}$.
\ee
For a subset $S\subset \BQ$, we write $S^*:=S\cup\{s\pm\epsilon|\ s\in S\}$, where we consider $\epsilon$ as infinitesimal. So we can identify $S_1$ with the set $\BR\cup \BQ^*$ in the following way :
For $r\in \BQ$ and $r'\in \BR\sm\BQ$,
\be
	\item $\{x=r\}\leftrightarrow r$;
	\item $\{l<x<r|\ l<r\}\leftrightarrow (r-\epsilon)$;
	\item $\{r<x<u|\ r<u\}\leftrightarrow (r+\epsilon)$; and
	\item $\{l<x<u|\ l<r'<u\}\leftrightarrow r'$.
\ee
Next we define a group-like structure on $S_1$. Define a plus-like operation $+^* : S_1\times S_1 \rightarrow \CP(S_1)$ as follows :
$$p_1 +^* p_2:=\{p|\ p\models (l_1+l_2<x<u_1+u_2),\ p_i\models l_i<x<u_i\},$$
and define a minus-like operation $-^* : S_1 \rightarrow S_1$ as follows :
$$(-^* p):=\{-u<x<-l|\ p\models l<x<u\}.$$
We define a composition of plus operation as follows :
$$(p_1 +^* p_2 )+^* p_3:=\bigcup_{p\in p_1 +^* p_2}\limits p+^* p_3,$$
and
$$p_1 +^* (p_2 +^* p_3):=\bigcup_{p\in p_2 +^* p_3}\limits p_1+^* p.$$
Then $+^*$ and $-^*$ is commutative, associate, and distributive. And for any $p_1,\cdots,p_k\in S_1$ and $k\ge 1$,
$$ |p_1+^* \cdots +^* p_k|\le 3$$
We write $p_1 -^* p_2 $ for $p_1+^* (-^*p_2)$. These two notions are naturally assigned to $\BR\cup\BQ^*$ and they are defined as follows :
\be
	\item 
	\be 
		\item If both $r_1$ and $r_2$ are in $\BR$ and let $r=r_1+r_2$, then
		\[r_1+^* r_2:= \left \{
		\begin{array}{ll}
         r & \mbox{if $\{r\}\in \BR\setminus\BQ$}\\
         r & \mbox{if $\{r\}\in \BQ$ and $r_1,r_2\in \BQ$}\\
        \{r-\epsilon,\ r, r+\epsilon\} & \mbox{if $r\in \BQ$ and $r_1,r_2\notin \BQ$}
        \end{array}
        \right.
        \];
		\item If $r_1\in \BR\setminus \BQ$ and $r_2=q\pm\epsilon \in \BQ^*$, then $r_1+^* r_2:=\{r_1+q\}$;
		\item If $r_1\in \BQ$ and $r_2=q\pm\epsilon \in \BQ^*$, then $r_1+^* r_2:=\{(r_1+q)\pm\epsilon\}$;		
		\item If $r_1=p\pm\epsilon$ and $r_2=q\pm\epsilon\in \BQ^*$, then $r_1+^* r_2:=\{(p+q)\pm\epsilon\}$;
		\item If $r_1=p\pm\epsilon$ and $r_2=q\mp\epsilon\in \BQ^*$, then $r_1+^* r_2:=\{(p+q)-\epsilon,(p+q),(p+q)+\epsilon\}$.
	\ee
	
	\item 
	\be 
		\item If $r_1\in \BR$, then $-^*r_1:=-r_1$;
		\item If $r_1=p\pm\epsilon \in \BQ^*$, then $-^*r_1:=-p\mp\epsilon$.
	\ee
\ee
Now we induce a group structure from $(S_1,+^*,-^*)$. Define a equivalence relation $\equiv_0$ on $S_1$,
$$p_1\equiv_0 p_2\ \mbox{iff}\ p_1-^* p_2 \subset \{0-\epsilon,\ 0,\ 0+\epsilon\},$$
and denote $[p]_0$ for $p\in S_1$ for the equivalence class.
Since $\{0-\epsilon,\ 0,\ 0+\epsilon\}$ is closed under $+^*$ and $-^*$, $+^*$ and $-^*$ are extended on $S_1/\equiv_0$. Then $(S_1/\equiv_0,+^*,-^*,[\tp(0)]_0 )$ is a group. Actually it is isomorphic to $(\BR,+,-,0)$.
\begin{theorem}
$(S_1/\equiv_0,+^*,-^*,[\tp(0)]_0)\cong (\BR,+,-,0)$.
\end{theorem}
\noindent And define a equivalence relation $\equiv_{\BZ}$ on $S_1$,
$$p_1\equiv_{\BZ} p_2\ \mbox{iff}\ p_1 -^* p_2 \subset \BZ^*,$$
and denote $[p]_{\BZ}$ for the equivalence class.
As same as $\equiv_0$, $\BZ^*$ is closed under $+^*$ and $-^*$ and $+^*$ and $-^*$ are extended on $S_1/\equiv_{\BZ}$. And $(S_1/\equiv_{\BZ},+^*,-^*,[\tp(0)]_{\BZ} )$ is isomorphic to $(\BR/\BZ,+,-,0)$ as groups.
\begin{theorem}
$(S_1/\equiv_{\BZ},+^*,-^*,[0]_{\BZ} )\cong (\BR/\BZ,+,-,0)$.
\end{theorem}
These two equivalences $\equiv_0$ and $\equiv_{\BZ}$ are defined on $\BR\cup \BQ^*$ and
$$(\BR\cup \BQ^*)/\equiv_0 \cong \BR\ \mbox{and}\ (\BR\cup \BQ^*)/\equiv_{\BZ} \cong \BR/\BZ.$$
Moreover we can define a multiplication-like operation $\times^*$ on $S_1$ as similar as the plus-like operation $+^*$. Note that for $r_0\in \BQ$ and $r\in \BR\cup\BQ^*$,

$
\begin{array}{c c l l}
\tp(r_0)\times^* \tp(r) & := & \tp(r_0r) & \mbox{if } r\in \BR \\
& & \tp(r_0q+\epsilon) & \mbox{if } r=q+\epsilon,\ q\in \BQ \\
& & \tp(r_0q-\epsilon) & \mbox{if } r=q-\epsilon,\ q\in \BQ \\
\end{array}
$

These plus-, minus-, multiplication-like operations make 
$$(S_1/\equiv_0,[\tp(0)]_0,[\tp(1)]_0,+^*,-^*,\times^*) \cong (\BR,0,1,+,-,\times),$$
and
$$(S_1/\equiv_{\BZ},[\tp(0)]_{\BZ},[\tp(1)]_{\BZ},+^*,-^*,\times^*) \cong (\BR/\BZ,0,1,+,-,\times).$$

\end{document}